\documentclass[reqno, 11pt, a4paper]{amsart}

\usepackage{fullpage}
\selectfont

\usepackage[utf8]{inputenc}
\usepackage[OT2,T1]{fontenc}
\usepackage[english]{babel}
\usepackage{amsmath}
\usepackage{amsfonts}
\usepackage{amssymb}
\usepackage{amsthm}
\usepackage{enumerate}
\usepackage{mathtools}
\usepackage[backref=page,pagebackref=true,hyperindex=true,bookmarks=true]{hyperref}
\usepackage{float}
\usepackage[ruled,vlined,boxed
]{algorithm2e}

\usepackage[dvipsnames]{xcolor}
\usepackage{graphicx}
\usepackage{tikz}
\usepackage{relsize}
\usepackage{stmaryrd}
\usepackage{bigints}
\usepackage[toc,page]{appendix}
\usepackage{graphicx}

\hypersetup{
  pdfauthor   = {E. Eid},
  pdftitle    = {Fast computation of hyperelliptic curve isogenies},
  pdfsubject  = {},
  pdfkeywords = {},
  colorlinks=true,
  breaklinks=true, urlcolor=blue, linkcolor=blue, citecolor=blue,
  bookmarksopen=true
}

\usetikzlibrary{cd}

\newcommand{\N}{\mathbb{N}}

\newcommand{\Kt}{K \llbracket t \rrbracket}
\newcommand{\kt}{k \llbracket t \rrbracket}

\newcommand{\OKt}{\mathcal{O}_{K} \llbracket t  \rrbracket}

\newcommand{\OK}{\mathcal{O}_{K}}

\newcommand{\vertIII}[1]{{\left\vert\kern-0.25ex\left\vert\kern-0.25ex\left\vert #1
    \right\vert\kern-0.25ex\right\vert\kern-0.25ex\right\vert}}

\newcommand{\softO}{\tilde O}
\newcommand{\cA}{\text{\rm A}}
\newcommand{\cM}{\text{\rm M}}
\newcommand{\cMM}{\text{\rm MM}}
\newcommand{\cC}{\text{\rm C}}
\newcommand{\cD}{\text{\rm D}}

\newtheorem{defi}{Definition}
\newtheorem{thm}[defi]{Theorem}
\newtheorem*{thm*}{Theorem}
\newtheorem{lem}[defi]{Lemma}
\newtheorem{prop}[defi]{Proposition}
\newtheorem{cor}[defi]{Corollary}

\theoremstyle{remark}
\newtheorem{remark}[defi]{Remark}

\title{Fast computation of hyperelliptic curve isogenies in odd characteristic}

\author[Eid]{\'{E}lie Eid}
\address{%
  \'{E}lie Eid, %
  Univ. Rennes, %
  CNRS, IRMAR - UMR 6625, F-35000
  Rennes, %
  France. %
}
\email{elie.eid@univ-rennes1.fr}

\begin{document}

\maketitle

\begin{abstract}
  Let $p$ be an odd prime number and $g \geq 2$ be an integer. We present an algorithm for computing explicit rational representations of isogenies between Jacobians of
  hyperelliptic curves of genus $g$ over an extension $K$ of the field of
  $p$-adic numbers $\mathbb{Q}_p$. It relies on an
  efficient resolution, with a logarithmic loss of $p$-adic precision, of a
  first order system of differential equations. 
\end{abstract}

\section{Introduction}
\label{sec:introduction}

After exploring elliptic curves in cryptography and their isogenies, and interest has been raised to their
generalizations. Researchers began to inspect principally polarized abelian
varieties, especially Jacobians of genus two and three curves and compute isogenies between them
\cite{cosset15,couezo15, milio, tian2020}. Their main interest was to calculate the
number of points of these varieties over finite fields \cite{gaudry12,
  lercier06, ballentine17} and more recently to instantiate isogeny-based cryptography schemes
\cite{EVYan19,costello20}. In this work, we concentrate on the problem of
computing explicitly isogenies between Jacobians of hyperelliptic curves over finite fields of odd characteristic, this will
be a generalization to~\cite{couezo15} and \cite{milio}.  \\
A separable isogeny between Jacobians of hyperelliptic curves of genus $g$
defined over a field $k$ is characterized by its so called rational
representation $($see Section~\ref{subsec:RationalRepresentation} for the
definition$)$; it is a compact writing of the isogeny
and can be expressed by $2g$ rational fractions defined over a finite
extension of $k$. These rational fractions are related. In fields of characteristic different from $2$, they can be determined by computing an
approximation of the solution $X(t) \in \kt{}^{g}$ of a first order non-linear
system of differential equations of the form
\begin{equation}
\label{eq:nonlinearsystem}
H \left( X(t) \right) \cdot X'(t) = G(t)
\end{equation}
where $H \! : \kt{}^{g} \rightarrow {M}_g\!\left( \kt \right)$ is a well chosen map and $G(t) \! \in \!  \kt{}^g$. This approach is a generalization of the elliptic curves case \cite{lava16} for which Equation~\eqref{eq:nonlinearsystem} is solved in dimension one. \\
Equation~\eqref{eq:nonlinearsystem} was first introduced in \cite{couezo15}
for genus two curves defined over finite fields of odd characteristic and
solved in \cite{kieffer20} using a well-designed algorithm based on a Newton
iteration; this
allowed them to compute $X(t)$ modulo $t^{O({\ell})}$ in the case of an
$(\ell , \ell )$-isogeny for a cost of $\softO (\ell )$ operations in $k$ then recover the
rational fractions that defines the rational representation of the
isogeny. This approach does not work when the characteristic of $k$ is positive and small compared to $\ell$, in which case divisions by $p$ occur and an error
can be raised while doing the computations. We take on this issue similarly as
in the elliptic curve case (\cite{lesi08, careidler20}) by lifting the problem
to the $p$-adics. We will always suppose that the lifted Jacobians are also
Jacobians for some hyperelliptic curves. It is relevant to assume this, even
though it is not the generic case when $g$ is greater than $3$ \cite{oort1986},
since it allows us to compute efficiently the rational representation of the
multiplication by an integer which in this case the lifting can be done arbitrarily. After this process, we need to analyze the loss of $p$-adic
precision in order to solve Equation~\eqref{eq:nonlinearsystem} without having a
numerical instability. We extend the result of \cite{lava16}, by proving that
the number of lost digits when computing an approximation of the solution of Equation~\eqref{eq:nonlinearsystem} modulo $t^{O(g\ell)}$, stays within $O\left( \log _p( g \ell) \right)$. Our
main theorem is the following.

\begin{thm*}
Let $p$ be a prime number. Let $K$ be a finite extension of $\mathbb{Q}_p$ and $\OK$ be its ring of integers. There exists an algorithm that takes as input:
\begin{itemize}
\item three positive integers $n$,$g$ and $N$,
\item a map $H \! :   \OKt{}^g \rightarrow M_g \! \left( \OKt \right)$ such that $H(0) \in  \text{GL}_g \left( \OK \right)$,
\item a vector $G(t) \in   \OKt{}^g$,
\end{itemize}
and, assuming that the differential equation
\begin{displaymath}
H \left( X(t) \right) \cdot X'(t) = G(t)
\end{displaymath}
admits a unique solution in $\left( t\OKt \right) ^g$, outputs an
approximation of this solution modulo $( p ^N , t^{n+1} )$ for a cost
$\softO \left( g^\omega n\right)$, where $\omega \in [2,3[$ is the exponent of matrix multiplication, at precision $O(p^M)$
with $M = \max ( N , 3) + \lfloor \log _ p (n) \rfloor$ if $p=2$,
$M=\max ( N , 2) + \lfloor \log _ p (n) \rfloor$ if $p=3$ and
$M= N + \lfloor \log _ p (n) \rfloor$ otherwise.
\end{thm*}

One can do a bit better for $p=2$ and $3$ if we follow the same strategy as \cite{lava16}, in
this case $M$ is equal to $ \max ( N , 2) + \lfloor \log _ p (n) \rfloor$ if
$p=2$ and $ N+ \lfloor \log _ p (n) \rfloor$ otherwise. For the sake of simplicity, we will not prove this here.

Note that this technique does not allow to compute isogenies in characteristic two for several reasons. First, the general equation of a hyperelliptic curve in characteristic two does not have the same form as in odd characteristic. Moreover, the map $H$ includes square roots of polynomials which implies that solving Equation~\eqref{eq:nonlinearsystem} will require to extract square roots at some point. However, it is well known that extracting square roots in an extension of $\mathbb{Q}_2$ is an unstable operation. Still, it is quite interesting to solve Equation~\eqref{eq:nonlinearsystem} for $p=2$ with the assumptions that we made in the main theorem, even thought this approach does not lead to the computation of isogenies between Jacobians of hyperelliptic curves.

\section{Jacobians of curves and their isogenies}
\label{sec:Review}

Throughout this section, the letter $k$ refers to a fixed field of
characteristic different from two. Let $\bar{k}$ be a fixed
algebraic closure of
$k$. In Section~\ref{subsec:abelianvarieties}, we briefly recall some basic
elements about principally polarized abelian varieties and $(\ell , \ldots , \ell)$-isogenies between them; the notion of rational representation is discussed in Section~\ref{subsec:RationalRepresentation}. Finally, for a given rational representation, we construct a system of differential equations that we associate with it.

\subsection{$(\ell , \cdots , \ell)$-isogenies between abelian varieties}
\label{subsec:abelianvarieties}
Let $A$ be an abelian variety of dimension $g$ over $k$ and $A^{\vee }$ be its dual.
To a fixed line bundle $\mathcal{L}$ on $A$, we associate the morphism $\lambda _ \mathcal{L}$ defined as follows
$$
\begin{array}{rcl}
\lambda _\mathcal{L} \, : \:  A & \longrightarrow & A ^\vee  \\
x & \longmapsto & t_x ^* \mathcal{L} \otimes \mathcal{L}^{-1}
\end{array}
$$
where $t_x$ denotes the translation by $x$ and $t_x^* \mathcal{L}$ is the pullback of $\mathcal{L}$ by $t_x$. \\
We recall from \cite{milne86} that \textit{a polarization} $\lambda$ of $A$ is
an isogeny $ \lambda : \, A \longrightarrow A^\vee$, that is a surjective homomorphism of abelian varieties of finite kernel, such that over $\bar{k}$, $\lambda$ is of the form $\lambda _ \mathcal{L}$ for some ample line bundle $\mathcal{L}$ on $A_{\bar{k}}:= A \otimes \text{Spec}(\bar{k})$.
When the degree of a polarization $\lambda$ of $A$ is equal to $1$, we say that $\lambda$ is a \textit{principal polarization} and the pair $(A, \lambda)$ is a \textit{principally polarized abelian variety}.
We assume in the rest of this subsection that we are given a principally polarized abelian variety $(A , \lambda)$. The \textit{Rosati involution} on the ring End$(A)$ of endomorphsims of $A$ corresponding to the polarization $\lambda$ is the map
$$
\begin{array}{rcl}
\text{ End}(A) & \longrightarrow & \text{End}(A) \\
\alpha & \longmapsto & \lambda ^{-1} \circ \alpha ^ \vee \circ \lambda.
\end{array}
$$
The Rosati involution is crucial for the study of the division algebra End$(A) \otimes \mathbb{Q}$, but for our purpose, we only state the following result.

\begin{prop}\cite[Proposition~14.2]{milne86}
\label{prop:NS}
For every $\alpha \in$ End$\,\!(A)$ fixed by the Rosati involution, there exists, up to algebraic equivalence, a unique line bundle $ \mathcal{L} _ A  ^\alpha $ on $A$ such that $\lambda _{ \mathcal{L} _A ^\alpha } = \lambda \circ \alpha$.
\end{prop}
In particular, taking $\alpha$ to be the identity endomorphism denoted ``$1$'', there exists a
unique line bundle $\mathcal{L}_A ^1 $ such that $\lambda _{ \mathcal{L}_A ^1
} = \lambda $.\\
Using Proposition~\ref{prop:NS}, we give the definition of an $(\ell , \ldots , \ell)$-isogeny.

\begin{defi}
\label{def:Isogeny}
Let $(A_1 , \lambda _1)$ and $(A_2 , \lambda _2)$ be two principally polarized abelian varieties of dimension $g$ over $k$ and $\ell \in \mathbb{N}^*$. An $(\ell ,  \ldots , \ell )$-isogeny $I$ between $A_1$ and $A_2$ is an isogeny $I : \, A_1 \longrightarrow A_2$ such that
$$ I ^* \mathcal{L}_{A_2} ^1 = \mathcal{L}_{A_1}^\ell,$$ where $\mathcal{L}_{A_1}^\ell$ is the unique line bundle on $A_1$ associated with the multiplication by $\ell$ map.
\end{defi}

We now suppose that $A$ is the Jacobian of a  genus $g$ curve $C$ over $k$. We will always make the assumption that there is at least one $k$-rational point on $C$. Let $r$ be a positive integer and fix $P \in C$. We define $C^{(r)}$ to be the symmetric power of $C$ and $j_P^{(r)}$ to be the map

$$
\begin{array}{rcl}
 \, C^{(r)} & \overset{j_P^{(r)}}\longrightarrow & A \simeq J(C) \\
(P_1 , \ldots , P_r) & \longmapsto & [P_1 + \cdots P_r - r P].
\end{array}
$$
If $r=1$ then the map $j_P^{(1)}$ is called the \textit{Jacobi map} with origin $P$.\\
We write $j^{(r)}$ for the map $j_P^{(r)}$. The image of $j^{(r)}$ is a closed
subvariety of $A$ which can be also written as $r$ summands of
$j^{(1)}(C)$. Let $\Theta$ be the image of $j^{(g)}$, it is a divisor on $A$
and when $P$ is replaced by another point, $\Theta$ is replaced by a translate.  We call $\Theta$ the theta divisor associated to $A$.

\begin{remark}
\label{rem:LinebundleTheta}
If $A$ is the Jacobian of a curve $C$ and $\Theta$ its theta divisor, then
$\mathcal{L}_A^1 = \mathcal{L}( \Theta )$, where $\mathcal{L}( \Theta )$ is the sheaf associated to the divisor $\Theta$. 
\end{remark}

Using Remark~\ref{rem:LinebundleTheta}, Definition~\ref{def:Isogeny} for Jacobian varieties gives the following
\begin{prop}
\label{prop:Characterization}
Let $\ell \in \mathbb{N}^*$, $A_1$ and $A_2$ be the Jacobians of two 
algebraic curves over $k$ and $\Theta _1$ and $\Theta _2$ be the theta divisors associated to $A_1$ and $A_2$ respectively. 
If an isogeny $I \,: \, A_1 \longrightarrow A_2$ is an $(\ell , \ldots , 
\ell )$-isogeny then $I^* \Theta _2$ is algebraically equivalent to $ 
\ell \Theta _1$.
\end{prop}

\begin{proof}
For all $x \in A_1$, the theorem of squares \cite[Theorem~5.5]{milne86} gives the following relation
$$ t_{\ell x} ^* \,  \mathcal{L}_{A_1}^1 \otimes \big ( \mathcal{L}_{A_1}^1 \big ) ^{-1} = t_x ^* \big ( \mathcal{L} _ {A_1} ^1 \big ) ^{\otimes \ell} \otimes \big ( (\mathcal{L}_{A_1} ^1 )^{\otimes \ell \,  } \big ) ^{-1}.$$
By Proposition~\ref{prop:NS}, the line bundle  $\mathcal{L}_{A_1} ^\ell$ is algebraically equivalent to $ \big ( \mathcal{L}_{A_1}^1 \big ) ^{\otimes \ell}$, therefore  $I^*\mathcal{L}_{A_2} ^1$ and $ \big ( \mathcal{L}_{A_1}^1 \big ) ^{\otimes \ell}$ are algebraically equivalent. By Remark~\ref{rem:LinebundleTheta},  $I^*\mathcal{L}_{A_2} ^1$ corresponds to $ I^* \Theta _2$ and $ \big ( \mathcal{L}_{A_1}^1 \big ) ^{\otimes \ell}$ corresponds to $ \ell \Theta _1$.
\end{proof}

\subsection{Rational representation of an isogeny between Jacobians of hyperelliptic curves}
\label{subsec:RationalRepresentation}
We focus on computing an isogeny between Jacobians of hyperelliptic curves.
Let $C_1$ $($resp. $C_2)$ be a genus $g$ hyperelliptic curve over $k$, $J_1$ $($resp. $J_2)$ be its associated Jacobian and $\Theta _1$ $($resp. $\Theta _2)$ be its theta divisor. We suppose that there exists a separable isogeny $I :  J_1 \longrightarrow J_2$. For $P \in C_1$, let $j_P \: : C_1 \longrightarrow J_1$ be the Jacobi map with origin $P$. Generalizing \cite[Proposition~4.1]{kieffer20} gives the following proposition

\begin{prop}
\label{prop:RationalRepresentation}
The morphism $I \circ j_P$ induces a unique morphism $I_P : \; C_1  \longrightarrow  C_2 ^{(g)} $ such that the following diagram commutes

\begin{center}
\begin{tikzcd}[column sep=large]
	& C_2 ^{(g)}  \\
  C_1 \arrow[ur, "I_P"] \arrow[dr, "I \circ j_P"'] &              \\
 & J_2 \arrow[uu ,"\simeq"']
        \end{tikzcd}
\end{center}

\end{prop}
$ $ \\
We assume that $C_1$ $($resp. $C_2)$ is given by the following singular model
$$ v ^2 = f_1(u) \quad (\text{resp. } y^2 = f_2(x))
$$
where $f_1$ $($resp. $f_2)$ is a polynomial of degree $2g+1$ or $2g+2$. Set
$Q=(u,v) \in C_1$ and $I_P (Q) = \{(x_1, y_1) , \ldots,  (x_g, y_g)\}$. We use
the Mumford's coordinates to represent the
element $I_{P} (Q)$: it is given by a pair of polynomials $(U(X) , V(X))$ such that 

$$U(X)= X^g + \mathbf{\sigma }_{1} X^{g-1}  + \cdots + \mathbf{\sigma}_g $$
where
$$
\mathbf{\sigma} _i = (-1)^{i} \sum \limits _{1 \leq j_1 < j_2 < \cdots < j_i \leq g} { x_{j_1} x_{j_2} \cdots x_{j_i}}$$
and

$$ V(X) = \mathbf{\rho}_{1}X^{g-1} +  \cdots + \mathbf{\rho}_g = \sum \limits _{j = 0} ^{g-1} {y_j \left( \prod \limits _{i= 0 , i\neq j} ^{g-1}  \dfrac{X- x_i}{x_j - x_i} \right) }. $$
The tuple $(\sigma _ 1, \cdots , \sigma _g , \rho _1 , \cdots , \rho _ {g})$ consists of rational fractions in $u$ and $v$ and it is called the \textit{rational representation} of $I$.

\begin{remark}
Since $I_{P} (u,-v) = -I_{P} (u,v)$, the functions $\mathbf{\sigma}_1 , \ldots ,  \mathbf{\sigma}_{g}$ can be seen as rational fractions in $u$ and have the same degree bounded by $\deg (\sigma _ 1)/2$. Moreover, the functions $\mathbf{\rho}_1 /v  , \ldots , \mathbf{\rho}_{g} /v $ can also be expressed as rational fractions in $u$ of degrees bounded by $\deg ( \rho _1) + 3 , \ldots , \deg (\rho_{g}) +3$ respectively.
\end{remark}

In order to determine the isogeny $I$, it suffices to compute its rational representation (because $I$ is a group homomorphism), so we need to have some bounds on the degree of the rational functions $\mathbf{\sigma}_1 , \ldots ,  \mathbf{\sigma}_{g}, \mathbf{\rho}_1 /v  , \ldots , \mathbf{\rho}_{g} /v $. In the case of an $(\ell , \ldots , \ell )$-isogeny, we adapt the proof of \cite[\S~6.1]{couezo15} in order to obtain bounds in terms of $\ell$ and $g$.

\begin{lem}
\label{lem:polardivisor}
Let $i \in \{ 1 , \ldots , g\}$. The pole divisor of  $\sigma _i$ seen as
function on $J_2$ is algebraically equivalent to $2 \Theta _ 2$. The pole divisor of $\rho _i$ seen as function
on $J_2$ is algebraically equivalent to $(2i+1) \Theta _ 2$ if $\deg
(f_2) = 2g+1$, and $(2i+2) \Theta _2$ otherwise.
\end{lem}

\begin{proof}
  This is a generalization of \cite[Lemma~4.25]{kieffer20}. Note that if
  $\deg (f_2) = 2g+1$, then $\sigma _ i$ has a pole of order one along the
  divisor $\{ ( R_1 , \ldots , R_{g-1} , \infty ) \, ; R_i \in C_2 \}$ which is algebraically equivalent to $2 \Theta _2$.
\end{proof}

\begin{lem}
\label{lem:Theta}\cite[Appendix]{matsusaka59}
The divisor $j_P(C_1)$ of $J_1$ is algebraically equivalent to
$\dfrac{\Theta _ 1^{g-1}}{(g-1)!}$ where $\Theta _1 ^{g-1}$ denotes the
$g-1$ times self intersection of the divisor $\Theta _1$.
\end{lem}

\begin{prop}
Let $\ell$ be a non-zero positive integer and $i \in \{ 1 , \ldots , g\}$. If
$I$ is an $(\ell , \ldots , \ell)$-isogeny, then the degree of $\sigma _i$
seen as a function on $C_1$ is bounded by $2g \ell$. The degree of $\rho _i$ seen as a function on $C_1$ is bounded by $(2i+1)g \ell$ if $\deg (f_2) = 2g+1$, and $(2i+2) g \ell$ otherwise.
\end{prop}

\begin{proof}
The degrees of $\sigma _ 1 , \ldots , \sigma _ g , \rho _1 , \ldots , \rho
_ g$ are obtained by computing the intersection of $j_P(C)$ with their pole
divisors. By Lemma~\ref{lem:polardivisor}, it suffices to show that $$ j_P(C) \cdot \Theta _ 2 = \ell g.$$
Since $I$ is an $(\ell , \ldots , \ell)$-isogeny, Proposition~\ref{prop:Characterization} gives that $I^* \Theta _ 2$ is algebraically equivalent to $\ell \Theta _ 1$. Moreover,
$$I^* \big ( I_P ( C) \big ) =  \big ( \vert \ker ( I) \vert \big )  \,    j_P(C) = l ^g j_P(C).$$
Using Lemma~\ref{lem:Theta}, we obtain
$$ I^* \big ( I_P(C) \big ) \cdot I^* \Theta _2  = gl^{g+1}.$$
As $$I^* \big ( I_P(C) \big ) \cdot I^* \Theta _2 =  \deg (I)  \, \big (  I_P(C) \cdot \Theta _2 \big ) = l^g (  I_P(C) \cdot \Theta _2 \big ),$$ the result follows.
\end{proof}

\subsection{Associated differential equation}
\label{subsec:diffeq}
We assume that char$(k) \neq 2$. We generalize \cite[\S~6.2]{couezo15} by constructing a differential system modeling the map $F_P= I \circ j_P $ of Proposition~\ref{prop:RationalRepresentation}.
The map $F_P$ is a morphism of varieties, it acts naturally on the spaces of holomorphic differentials $H ^0 ( J_2 ,  \Omega^1 _{J_2} )$ and $ H ^0 ( C_1 ,  \Omega^1 _{C_1} )$ associated to $J_2$ and $C_1$ respectively, this action gives a map $$ F_P ^*  \, : \, H ^0 ( J_2 ,  \Omega^1 _{J_2} ) \longrightarrow H ^0 ( C_1 ,  \Omega^1 _{C_1} ).$$
A basis of $H ^0 ( C_1 ,  \Omega^1 _{C_1} )$ is given by
$$ B_1 = \left\{ u ^i \dfrac{du}{v} \, ; i \in \{ 0,\ldots , g-1\} \right\} .$$
The Jacobi map of $C_2$ induces an isomorphism between
the spaces of holomorphic differentials associated to $C_2$ and $J_2$, so $H
^0 ( J_2 ,  \Omega^1 _{J_2} )$ is of dimension $g$, it can be identified with
the space $ H ^0 ( C_2^g ,  \Omega^1 _{C_2^g} ) ^ {S_n}$ (here the symmetric group $S_n$ acts naturally on the space $ H ^0 ( C_2^g ,  \Omega^1 _{C_2^g} )$). With this identification, a basis of $ H ^0 ( J_2 ,  \Omega^1 _{J_2} )$ is chosen to be equal to
$$ B_2 = \left\{  \sum \limits _{j=1}^ g { x_j ^i \dfrac{dx_j}{y_j}} \, ; \, i \in \{ 0, \ldots , g-1\} \right\}.$$
Let $(m_{ij})_{0 \leq i,j \leq g} \in  \text{ GL}_g(\bar{k})$ be the matrix of $F_P^*$ with respect of these two bases, we call it the \textit{normalization matrix}.
Let $Q=(u_Q,v_Q) \in C_1$ be a non-Weierstrass point different from $P$ and
$I_{P} (Q) = \{R_1 ,\ldots , R_g\}$ such that $I_P(Q)$ contains $g$ distinct points and does not contain neither a point at infinity nor a Weierstrass point.
The points $R_i$ may be defined over an extension $k'$ of $k$ of degree equal to $O(g)$. Let $t$ be a formal parameter of $C_1$ at $Q$, then we have the following diagram

\begin{center}
\begin{tikzcd}[column sep=large]
	Spec \big ( k' \llbracket t \rrbracket \big ) \arrow[rr, "t \mapsto ( R_i(t))_i"] \arrow[dd] & & C_2 ^g  \arrow[dd]  \\
	 & &\\
  C_1 \arrow[rr, "I_P"] & & C_2 ^{(g)}

\end{tikzcd}
\end{center}
This gives the differential system
\begin{equation}
\label{eq:diffsys}
\left \{
\begin{array}{ccccccc}
\dfrac{dx_1}{y_1}& + & \cdots & + & \dfrac{dx_g}{y_g} & = & \big ( {m_{11} + m_{12} \cdot u + ... + m_{1g}\cdot u^{g-1}} \big ) \dfrac{du}{v}\,, \\
& &  & & &  &  \\
\dfrac{x_1 \cdot dx_1}{y_1}& + & \cdots & + & \dfrac{x_g \cdot dx_g}{y_g} & = & \big ( {m_{21} + m_{22} \cdot u +  ... + m_{2g}\cdot u^{g-1}} \big ) \dfrac{du}{v}\,,  \\
& & \vdots & & &  & \vdots  \\
\dfrac{x_1 ^{g-1} \cdot dx_1}{y_1}& + & \cdots & + & \dfrac{x_g^{g-1} \cdot dx_g}{y_g} & = & \big ( {m_{g1} + m_{g2} \cdot u +  ... + m_{gg}\cdot u^{g-1}} \big ) \dfrac{du}{v}\,, \bigskip \\
y_1^2 = f_2(x_1), & & \cdots & &, \, y_g^2 = f_2(x_g)\,. &
\end{array}
\right.
\end{equation}

Equation~\eqref{eq:diffsys} has been initially constructed and solved in \cite{couezo15} for $g = 2$. In this case, the normalization matrix and the initial condition $(x_1(0) , x_2(0))$ are computed using algebraic theta functions.
In a more practical way, we refer to \cite{kieffer20} for an easy computation of the initial condition $(x_1(0) , x_2(0))$ of Equation~\eqref{eq:diffsys} and for solving the differential system using a Newton iteration. However, in this case, the normalization matrix is determined by differentiating modular equations.
There is a slight difference in Equation~\eqref{eq:diffsys} between the two
cases, especially $x_1(0)$ and $x_2(0)$ are different in the first, and equal in the second. Let $H$ be the $g$-squared matrix defined by $$H(x_1,\ldots x_g) = \left( x_j^{i-1}\dfrac{1}{y_j} \right) _{1 \leq i,j \leq g }.$$ We suppose that $g=2$. If the initial condition $(x_1(0), x_2(0))$ of Equation~\eqref{eq:diffsys} satisfies $x_1(0) \neq x_2(0)$, then the matrix $H ( x_1(0), x_2(0)  )$ is invertible in ${M}_2 \! \left( k \right)$. Otherwise, its determinant is equal to zero. \\
More generally, we prove that with the assumptions that we made on $Q,R_1,R_2, \ldots R_{g-1}$ and $R_g$, the matrix $H ( x_1(0), \ldots ,  x_g(0)  )$ is invertible in ${M}_g \! \left( k \right)$. Let $t$ be a formal parameter, $Q(t)$ the formal point on $C_1 \left( k\llbracket t \rrbracket \right)$ that corresponds to $t = u - u_Q$ and $\{R_1(t) , \ldots , R_g(t)\}$ the image of $Q(t)$ by $I_{P}$, then Equation~\eqref{eq:diffsys} becomes
\begin{equation}
\label{eq:diffsys1}
H \! \left( X(t) \right) \cdot  X'(t) = G(t)
\end{equation}
where $X(t) = (x_1 (t) , \ldots  , x_g(t))$ and $G(t) = v ^{-1}\left( \sum \limits _{i= 1}^g {m_{ij} {u^{i-1}}} \right) _ {1 \leq j \leq g}$. Thus we have the following proposition
\begin{prop}
\label{prop:Determinant}
The matrix $H \! \left( X(t) \right)$ is invertible in ${M}_g \! \left( k \llbracket t \rrbracket \right)$.
\end{prop}

\begin{proof}
The matrix $H \! \left( X(t) \right)$ is sort of a generalization of the Vandermonde matrix, its determinant is given by
$$ \det \left( H \! \left( X(t) \right) \right) = \dfrac{\prod \limits _{ 1 \leq i < j \leq g} {\left( x_j(t) - x_i(t) \right)} }{\prod \limits _{i=1} ^ g {y_i (t) } }$$
which is invertible in ${M}_g \! \left( k \llbracket t \rrbracket \right)$ because $x_i(0) \neq x_j(0)$  for all $i,j \in \{ 1 , \ldots , g\}$ such that $i\neq j$.
\end{proof}

\section{Fast resolution of systems of $p$-adic differential equations}
In this section, we give a proof of the main theorem by solving efficiently the nonlinear system of differential equations \eqref{eq:nonlinearsystem} in an extension of $\mathbb{Q}_p$ for all prime numbers $p$ even though it is not useful for computing isogenies for $p=2$. In Section~\ref{subsec:computationalmodel}, we introduce the computational model that we use in our algorithm exposed in Section~\ref{subsec:algorithm} and the proof of its correctness is presented in Section~\ref{subsec:precisionanalysis}.\\
Throughout this section the letter $p$ refers to a fixed prime number and $K$ corresponds to a fixed finite extension of $\mathbb{Q}_p$. We denote by $\upsilon _p$ the unique normalized extension to $K$ of the $p$-adic valuation. We denote by $\OK$ the ring of integers of $K$, $\pi \in \OK$ a fixed uniformizer of $K$ and $e$ the ramification index of the extension $K/\mathbb{Q}_p$. We naturally extend the valuation $\upsilon _ p$ to quotients of $\OK$, the resultant valuation is also denoted by $\upsilon _ p$.
\subsection{Computational model}
 \label{subsec:computationalmodel}
From an algorithmic point of view, $p$-adic numbers behave like real numbers: they are defined as infinite sequences of digits that cannot be handled by computers. It is thus necessary to work with truncations. For this reason, several computational models were suggested to tackle these issues (see \cite{caruso17} for more details). In this paper, we use the fixed point arithmetic model at precision $O(p^M)$, where $M \in \mathbb{N}^*$, to do computations in $K$. More precisely, an element in $K$ is represented by an interval of the form $a+ O(p^M)$ with $a \in \OK/\pi ^{eM}\OK$. We define basic arithmetic operations on intervals in an elementary way
\begin{align*}
\big( x + O(p^M) \big) \pm \big( y + O(p^M) \big) & = (x \pm y) + O(p^M)\,, \\
\big( x + O(p^M) \big) \times \big( y + O(p^M) \big) & = xy + O(p^M)\,.
\end{align*}
For divisions we make the following assumption: for $x, y \in
\OK/\pi^{eM}\OK$, the division of $x + O(p^M)$ by $y + O(p^M)$ raises an error if $\upsilon _p(y) > \upsilon _p(x)$, returns $0 + O(p^M)$ if $x = 0$ in $\OK/\pi^{eM}\OK$ and returns any representative $z + O(p^M)$ with the property
$x = yz$ in $\OK/\pi^{eM}\OK$ otherwise.

\subsubsection*{Matrix computation}
We extend the notion of intervals to the $K$-vector space $M_{n,m} (K)$: an element in
$M_{n,m}\! \left(K\right)$ of the form $A+ O(p^M)$ represents a matrix
$\left( a_{ij} + O(p^M) \right)_{ij}$ with
$A = (a_{ij}) \in M_{n,m}\! \left(\OK/\pi ^{eM}\OK \right)$. Operations in
$M_{n,m}\! \left(K \right)$ are defined from those in $K$:
$$
\left( A + O(p^M) \right) \pm \left( B + O(p^M) \right) = ( A \pm B  ) + O(p^M),$$
 $$
\left( A + O(p^M) \right) \cdot \left( B + O(p^M) \right) = ( A  \cdot  B  ) + O(p^M).$$
For inversions, we use standard Gaussian elimination.
\begin{lem}\cite[Proposition~1.2.4 and Th\'eor\`eme~1.2.6]{vaccon15} 
\label{prop:Jordan}
Let $A$ be an invertible matrix in $M_n(\mathcal{O}_K)$ with entries known up to precision $O(p^M)$. The Gauss-Jordan algorithm computes the inverse $A^{-1}$ of $A$ with entries known with the same precision as those of $A$ using $O(n^3)$ operations in $K$. 
\end{lem}

\subsection{The algorithm}
\label{subsec:algorithm}
Let $g$ be a positive integer, $\Kt$ be the ring of formal series over $K$ in $t$. We denote by ${M}_g\! \left(k\right)$ the ring of square matrices of size $g$ over a field $k$. Let $\mathrm{f} = \big ( f_{ij} \big ) _{  i,j } \! \!  \in {M}_g \! \left ( \Kt \right)$ and $H_\mathrm{f}$ be the map defined by
$$
\begin{array}{rcl}
 \big ( t \Kt \big )^g  & \overset{H_\mathrm{f}} {\xrightarrow{\hspace*{2cm}}} & {M}_g \!  \left( \Kt \right)\, \\
\big ( x_1(t) , \ldots , x_g (t) \big ) & \xmapsto{\hspace*{1.5cm}} & \Big ( f_{ij} \big ( x_i(t)  \big ) \Big ) _{ij}\,.
\end{array}
$$
Given $ \mathrm{f}  \in {M}_g \! \left( \Kt \right)$ and $ G = ( G_1, \ldots , G_g) \in  \Kt{} ^g$, we consider the following differential equation in $ X = (x_1,\ldots , x_g)$,
\begin{equation}
\label{eq:nonlinearequation}
H_\mathrm{f} \circ X \cdot X' = G.
\end{equation}
We will always look for solutions of \eqref{eq:nonlinearequation} in $\big (
t\Kt \big ) ^g$ in order to ensure that $H_\mathrm{f} \circ X$ is well
defined. We further assume that $H_\mathrm{f}(0)$ is invertible in ${M}_g \! \left( K \right)$.

\begin{remark}
Up to a change of variables, the differential system \eqref{eq:diffsys1} fulfills all the assumptions of Equation~\eqref{eq:nonlinearequation}.
\end{remark}

The next proposition guarantees the existence and the uniqueness of a solution of the differential equation \eqref{eq:nonlinearequation}.

\begin{prop}
\label{prop:solutionexistence}
Assuming that $H_\mathrm{f}(0)$ is invertible in ${M}_g \!  \left( K \right)$, the system of differential equations \eqref{eq:nonlinearequation} admits a unique solution in $  \Kt{} ^g$.
\end{prop}

\begin{proof}
We are looking for a vector $X(t) = \sum  \limits _{n=1}^\infty {X_n t^n}$ that satisfies Equation~\eqref{eq:nonlinearequation}.
Since $X(0)= 0$ and $H_\mathrm{f} (0)$ is invertible in $ \Kt {} ^g$, then $H
_\mathrm{f} \big (X (t)\big )$ is invertible in $ M_{g} \!  \left( \Kt
\right)$. So Equation~\eqref{eq:nonlinearequation} can be written as
\begin{equation}
\label{eq:existanceproof}
 X '(t) = \big ( H _\mathrm{f}( X (t)) \big )^{-1} \cdot G(t).
\end{equation}
Equation~\eqref{eq:existanceproof} applied to $0$, gives the non-zero vector $X_1$. Taking the $n$-derivative of Equation~\eqref{eq:existanceproof} with respect to $t$ and applying the result to $0$, we observe that the coefficient $X_n$ only appears on the hand left side of the result, so each component of $X_n$ is a polynomial in the components of the $X_i$'s for $i<n$ with coefficients in $K$. Therefore, the coefficients $X_n$ exist and are all uniquely determined.
\end{proof}

We construct the solution of Equation~\eqref{eq:nonlinearequation} using a Newton scheme. We recall that for $Y = (y_1, \ldots ,y_g) \in \Kt {}^g$, the differential of $H_\mathrm{f}$ with respect to $Y$ is the function
\begin{equation}
\label{eq:chainformula}
\begin{array}{rcl}
dH_\mathrm{f} (Y) \, : \,  \Kt{} ^g & \longrightarrow & {M}_g \! \left( \Kt \right)\, \\
h & \longmapsto & dH_\mathrm{f}(Y)(h) = \left( { {f'_{ij}}} \left( y_i \right) \cdot  h_i \right)_{1\leq i,j \leq g}\,.
\end{array}
\end{equation}

We fix $m \in \N$ and we consider an approximation $X_m$ of $X$ modulo $t^m$. We want to find a vector $h \in \left( t^m \, \Kt \right)^g $, such that $X_m +h$ is a better approximation of $X$. We compute
$$
  H_\mathrm{f} \left( X_m + h \right)  =   H_\mathrm{f} \left( X_m \right)+ dH_\mathrm{f}(X_m) (h) \pmod {t^{2m}}\,.
$$
Therefore we obtain the following relation
\begin{multline*}
 H_\mathrm{f} \left( X_m + h \right) \cdot \left( X_m + h \right) ' - G  = \\
H_\mathrm{f} \left( X_m \right) \cdot X_m' + H_\mathrm{f} \left( X_m \right) \cdot h' + dH_\mathrm{f}(X_m) (h) \cdot X_m' - G \pmod {t^{2m-1}}\,.
\end{multline*}
So we look for $h$ such that
\begin{equation}
\label{eq:newtoneq}
 H_\mathrm{f} \left( X_m  \right) \cdot h'  + dH_\mathrm{f}(X_m) (h) \cdot X_m'  =  -H_\mathrm{f} \left( X_m \right)\cdot X_m' +G \pmod {t^{2m-1}}\,.
\end{equation}
It is easy to see that the left hand side of Equation~\eqref{eq:newtoneq} is equal to $\left( (H_\mathrm{f} \left( X_m \right) \cdot h \right)^{'}$, therefore integrating each component of Equation~\eqref{eq:newtoneq} and multiplying the result by $ \left(H_\mathrm{f} \left( X_m \right) \right)^{-1} $ gives the following expression for $h$
\begin{equation}
\label{eq:newtoniteration}
 h= \left(H_\mathrm{f} \left( X_m\right) \right) ^{-1} \,  \int {\left( G - H_\mathrm{f} \left( X_m \right)\cdot X_m' \right)\, dt} \pmod {t^{2m}},
\end{equation}
where $\int Y dt$, for $Y \in  \Kt {} ^g$, denotes the unique vector $I \in  \Kt {}^g$ such that $I' = Y$ and $I(0)=0$.\\
This formula defines a Newton operator for computing an approximation of the solution of Equation~\eqref{eq:nonlinearequation}. Reversing the above calculations leads to the following proposition.
\begin{prop}
\label{prop:newtoncase1}
We assume that $H_\mathrm{f}(0)$ is invertible in ${M}_g \! \left( K \right)$. Let $m>0$ be an integer, $n=2m$ and $X_m \in  \Kt {}^g$ a solution of Equation~\eqref{eq:nonlinearequation} mod $t^m$. Then,
$$
X_n = X_m +  \left( H_\mathrm{f} \left( X_m \right) \right)^{-1} \int {\left(G - H_\mathrm{f} \left( X_m \right) \cdot X_m'\right)\, dt}
$$
is a solution of Equation~\eqref{eq:nonlinearequation} mod $t^{n+1}$.
\end{prop}

It is straightforward to turn Proposition~\ref{prop:newtoncase1} into an algorithm that solves the nonlinear system~\eqref{eq:nonlinearequation}. We make a small optimization by integrating the computation of $H_\mathrm{f}(X)^{-1}$ in the Newton scheme.

\begin{figure}[H]
  \begin{center}
    \parbox{0.85\linewidth}{%
      \begin{footnotesize}\SetAlFnt{\small\sf}%
        \begin{algorithm}[H]%
          \caption{Differential Equation Solver} %
          \label{algo:DiffSolve}%
          \SetKwInOut{Input}{Input} %
          \SetKwInOut{Output}{Output} %
          \SetKwProg{DiffSolve}{\tt DiffSolve}{}{}%
          \DiffSolve{$(G,\mathrm{f},n , g)$}{
            \Input{$G, \mathrm{f} \mod {t^n}$ such that $H_\mathrm{f}(0)$ is invertible in $M_g \! \left( K \right)$. }
            \Output{The solution $X$ of Equation~\eqref{eq:nonlinearequation}$\mod {t^{n+1}}$, $H_\mathrm{f}\left(X \right) \mod t^{\lceil n/2 \rceil}$} \BlankLine %
            \If{$n =0$}
            {%
            \KwRet{$ 0 \mod t, \, H_\mathrm{f} (0)^{-1} \mod t$}
            }%
            {}
            $ m := \lceil \frac{n-1}{2} \rceil$\;%

            $X_m,\, H_m:= $ \texttt{DiffSolve}{$(G,\mathrm{f},m ,g)$}\;%

            $H_n := 2H_m - H_m \cdot H_\mathrm{f}(X) \cdot H_m \mod t^{m+1}$

            \KwRet{$X_m +  H_n \mathlarger{\int} {\left(G - H_\mathrm{f} \left( X_m \right)\cdot X_m'\right)\, dt}\; \mod {t^{n+1}}$}
          }
        \end{algorithm}
      \end{footnotesize}
    }
  \end{center}
\end{figure}

According to Proposition~\ref{prop:newtoncase1},
Algorithm~\ref{algo:DiffSolve} runs correctly when its entries are given with an infinite $p$-adic precision; however it could stop working if we use the fixed point arithmetic model.  The next theorem
guarantees its correctness in this type of models.

\begin{thm}
\label{thm:mainthm}
Let $n,g \in \mathbb{N}$, $N\in \frac{1}{e}\mathbb{Z}^*, G \in \OKt ^g$ and
$\mathrm{f} \in M_g\! \left( \OKt\right)$. We assume that $H_\mathrm{f}(0)$ is
invertible in $M_g \! \left( \OK \right)$ and that the components of the
solution of Equation~\eqref{eq:nonlinearequation} have coefficients in $\OK$. When the procedure $\textrm{\tt DiffSolve}$ runs
with fixed point arithmetic at precision $O(p^M)$, with
$M = \max ( N , 3) + \lfloor \log _ p(n) \rfloor$ if $p=2$,
$M = \max ( N , 2) + \lfloor \log _ p(n) \rfloor$ if $p=3$ and
$M = N + \lfloor \log _ p(n) \rfloor$ otherwise. All the computations are done
in $\OK$ and the result is correct at precision $O(p^N)$.
\end{thm}

We give a proof of Theorem~\ref{thm:mainthm} at the end of Section~\ref{subsec:precisionanalysis}. Right now, we concentrate on the complexity of Algorithm~\ref{algo:DiffSolve}. Let $\cMM(g,n)$ be the number of arithmetical operations required to compute the product of two $g \times g$ matrices containing polynomials of degree $n$ with coefficients in $K$ and $\cM(n) := \cMM (1,n)$, therefore $\cM(n)$ is the number of arithmetical operations required to compute the product of two polynomials of degree $n$. According to \cite[Chapter~8]{Bostan17}, the two functions $\cM(.)$ and $\cMM(g,.)$ are related by the following formula
\begin{equation}
\label{eq:relationcomplexity}
 \cMM(g , n) = O \left( g^\omega \cM(n) \right),
\end{equation}

where $\omega \in [2,3[$ is the exponent of matrix multiplication. Furthermore, we denote by $\cC _{H} (n)$ the algebraic complexity for
computing $H \circ X \mod t^n$ for any map $H : \Kt ^g \rightarrow M_g \!
\left( \Kt \right)$. We assume that $\cM (n)$ and $\cC _ H(n)$ satisfy the superadditivity hypothesis
\begin{equation}
\label{eq:superadditivity}
\begin{array}{cccc}
\cM (n_1 + n_2 ) & \geq &  \cM (n_1) + \cM (n_2),&\\
 \cC_H (n_1 + n_2 ) & \geq &  \cC _H  (n_1) + \cC_H(n_2), & \forall n_1 , n_2 \in \mathbb{N}.
\end{array}
\end{equation}
For instance, when $H$ is given by a matrix $(f_{ij})_{i,j}$ such that
$f_{ij}$ is an univariate polynomial of degree $d$ for every
$i,j \in \{1,\ldots , g\}$, then $\cC _{H} (n) = O\left(g^2 d \cM(n) \right)$.
\begin{remark}
In the situation of Equation~\eqref{eq:diffsys}, the map $H$ includes univariate
rational fractions of radicals of degree $O(g)$; in this case, we compute
$y_1^2 , \ldots , y_g^2 \mod t^n$, we use a Newton scheme to compute
$y_1^{-1} , \ldots , y_g^{-1} \mod t^n$, then we compute
$ x_i y_i^{-1},x_i^2 y_i^{-1}, \ldots x_i^{g-1}y_i^{-1} \mod t^n$ for $i=1, \ldots,g$. The algebraic complexity $\cC_H(n)$ is therefore equal to 
$\cC _ H ( n) = O\left( g^2 \cM(n) \right)$.

\end{remark}
\begin{prop}
\label{prop:complexity}
Algorithm~\ref{algo:DiffSolve} performs $O \left( \cMM ( g , n ) + \cC _{H_\mathrm{f}}(n) \right)$ operations in $K$.
\end{prop}

\begin{proof}
The complexity of computing $ H_\mathrm{f}(0) ^{-1}$ is at most $O(g^\omega)$ operations in $K$. Let $\cD$ denote the algebraic complexity of Algorithm~\ref{algo:DiffSolve}, then we have the following relation
$$
\cD (n)  \leq  \cD \left( \left\lceil \dfrac{n-1}{2} \right\rceil \right) + O \left( \cMM ( g , n ) + \cC _{H_\mathrm{f}}(n) \right).$$
Noticing that $g$ is fixed and using Eqs.~\eqref{eq:relationcomplexity} and \eqref{eq:superadditivity}, we find $ \cD (n) = O \left( \cMM ( g , n ) + \cC _{H_\mathrm{f}}(n) \right)$ and the result is proved.
\end{proof}

\begin{cor}
When performed with fixed point arithmetic at precision $O(p^M)$, the bit complexity of Algorithm~\ref{algo:DiffSolve} is $O \left(  \left(\cMM ( g , n ) + \cC _{H_\mathrm{f}}(n) \right) \cdot \cA (K;M) \right)$ where $\cA (K;M)$ denotes an upper bound on the bit complexity of the arithmetic operations in $\OK / \pi ^{eM} \OK$.
\end{cor}

\subsection{Precision analysis}
\label{subsec:precisionanalysis}
The goal of this subsection is to prove Theorem~\ref{thm:mainthm}. The proof relies on the the theory of "differential precision" developed in \cite{carova14,carova15}. We follow the same strategy of \cite{careidler20,lava16}. \\
Let $g$ be a fixed positive integer. We study the solution $X$ of Equation~\eqref{eq:nonlinearequation} when $G$ varies, with the assumption $H_\mathrm{f}(0)$ is invertible in ${M}_g \! \left( \OK \right)$. Proposition~\ref{prop:solutionexistence} showed that Equation~\eqref{eq:nonlinearequation} has a unique solution $X\! \left( G\right) \in  \Kt  ^g$. Moreover, if we examine the proof of Proposition~\ref{prop:solutionexistence}, we see that the $n+1$ first coefficients of the vector $X\! \left(G\right)$ depends only on the first $n$ coefficients of $G$. This gives a well-defined function
$$
\begin{array}{rcl}
X_n \, : \; \left( \Kt / \left( t^n \right) \right) ^g & \longrightarrow &  \left( t\Kt / \left( t^{n+1} \right) \right) ^g \\
G & \longmapsto & X\! \left(G \right)
\end{array}
$$
for a given positive integer $n$. In addition, the proof of Proposition~\ref{prop:solutionexistence} states that for $G \in \left( \Kt / \left( t^n \right) \right) ^g$, $X_n \! \left(G \right)$ can be expressed as a polynomial in $ G(0) , G'(0), \ldots , G^{(n-1)}(0)$ with coefficients in $K$, therefore $X_n$ is locally analytic.
\begin{prop}
\label{prop:differentialX}
For $G \in  \left( \Kt / \left( t^n \right) \right) ^g$, the differential of $X_n$ with respect to $G$ is the following function
$$
\begin{array}{rcl}
dX_n \! \left(G \right) \, : \;  \left( \Kt / \left( t^n \right) \right) ^g & \longrightarrow &  \left( t\Kt / \left( t^{n+1} \right) \right) ^g \\
\delta G & \longmapsto & \left( H_\mathrm{f} \left( X_n \! \left(G\right) \right) \right) ^{-1} \cdot \mathlarger{\int} \delta G.
\end{array}
$$
\end{prop}

\begin{proof}
We differentiate the equation $  H _ \mathrm{f} \! \left( X_n \!  \left( G \right) \right) \cdot X_n \! \left(G \right)' = G$ with respect to $G$, we obtain the following relation
\begin{equation}
\label{eq:proofdifferentialX}
H _ \mathrm{f} \! \left( X_n \! \left( G \right) \right) \cdot \big ( dX_n (G)(\delta G) \big )^{'} + d H_\mathrm{f} \! \left( X_n \! \left(G \right) \right)\!  \left(dX_n(G)(\delta G)\right) \cdot X_n(G)' = \delta G
\end{equation}
where $d H_\mathrm{f} \! \left( X_n(G) \right)$ is the differential of $H_\mathrm{f}$ at $X_n(G)$ defined in \eqref{eq:chainformula}. Making use of the relation
$$
\big ( \left( H _\mathrm{f} \! \left( X_n(G) \right) \right) \cdot dX_n(G)(\delta G) \big ) ^{'} = H _ \mathrm{f} \! \left( X_n \! \left( G \right) \right) \cdot \big ( dX_n (G)(\delta G) \big )^{'} + d H_\mathrm{f}\! \left( X_n(G) \right) \! \left(dX_n(G)(\delta G) \right) \cdot X_n(G)' ,
$$
Equation~\eqref{eq:proofdifferentialX} becomes
$$
\big ( H_\mathrm{f}\!  \left( X_n(G)\right)  \cdot dX_n(G)(\delta G) \big ) ^{'} = \delta G.
$$
Integrating the above relation and multiplying by $ \left( H_\mathrm{f} \! \left( X_n(G) \right) \right)^{-1}$ we get the result.
\end{proof}
We now introduce some norms on $\left( \Kt / \left( t^n \right) \right) ^g$ and $\left( t\Kt / \left( t^n \right) \right) ^g$. We set $ E_n = \left( \Kt / \left( t^n \right) \right) ^g$ and $F_n =\left( t\Kt / \left( t^{n+1} \right) \right) ^g$; for instance, $X_n$ is a function from $E_n$ to $F_n$.\\
First, we equip the vector space $K_n := \Kt / \left( t^n \right)$ with the usual Gauss norm
$$ \Vert a_0 + a_1 t + \cdots + a_{n-1} \Vert _{K_n} = \max \left( \left \vert a_0 \right \vert ,  \left \vert a_1 \right \vert, \ldots ,  \left \vert a_{n-1} \right \vert \right).$$
We equip ${M}_g \! \left( \Kt /(t^n) \right)$ with the induced norm: for every $A= \left( a_{ij}(t)  \right) _ {ij} \in {M}_g \!  \left( \Kt /(t^n) \right)$,
$$
\left \Vert A \right \Vert = \max \limits _{i} \sum \limits _{j=1}^{g} {  \Vert  a_{ij}(t)  \Vert _ {K_n} } .
$$
We endow $F_n$ with the norm obtained by the restriction of the induced norm $\Vert . \Vert$ on $F_n$: for every $X=  \left( x_i(t) \right) _{i} \in F_n,$
$$
\left\Vert x \right\Vert _ {F_n}  =  \max \limits _{i} \left\Vert x_i(t) \right\Vert _ {K_n} .
$$
In the other hand, we endow $E_n$ with the following norm: for every $X =  \left( x_i(t) \right) _i \in E_n,$
$$
\left\Vert x \right\Vert _{E_n} =  \Vert \int x \,  \Vert _{F_n} = \max \limits _{ i } \Vert  \int  x_i(t) \, \Vert _ {K_n}.$$
\begin{lem}
\label{lem:compatible}
The induced norm on ${M}_g \!  \left( \Kt /(t^n) \right)$ is compatible with the norm on $F_n$, in other words we have
$$ \Vert A \, x \Vert _ {F_n} \leq \Vert A \Vert  \, \Vert x\Vert _{F_n}
$$ for all $A \in {M}_g \!  \left( \Kt /(t^n) \right)$ and $ x \in F_n$.
\end{lem}

\begin{proof}
The result follows immediately from the sub-multiplicativity of the norm $\Vert . \Vert _ {K_n}$.
\end{proof}

\begin{lem}
\label{lem:isometry}
Let $G \in \left( \OKt / \left( t^n \right) \right) ^g$. We assume that $X_n \! \left(G \right) \in \left( t\OKt / \left( t^n \right) \right) ^g$, then $dX_n\! \left( G \right) \, : \; E_n \longrightarrow F_n$ is an isometry.
\end{lem}

\begin{proof}
The assumptions $X_n(G) \in \left( t\OKt / \left( t^n \right) \right) ^g$ and $H_\mathrm{f}(0) \in \text{GL}_g \big (\OK \big )$ guarantee the invertibility of $H _\mathrm{f} \! \left( X_n(G) \right)$ in ${M}_g \!  \left( \OKt \right)$. Therefore, the norm $\left \Vert H_\mathrm{f} \! \left( X_n(G) \right) \right \Vert$ is equal to one. It follows from Lemma~\ref{lem:compatible} that the product $\left( H _\mathrm{f} \! \left( X_n(G) \right) \right) \cdot \mathlarger{\int} \delta G$ and $\mathlarger{\int} \delta G$ have the same norm on $F_n$, which is equal to $\Vert \delta G \Vert _ {E_n}$.
\end{proof}

We define the following function:
$$
\begin{array}{rcl}
\tau _ n \, : \, F_n \times E_n & \longrightarrow &  \text{Hom}(E_n, F_n)\\
( X \, , \, G) & \longmapsto & \left( \delta G \mapsto \left( H_\mathrm{f} \! \left( X \right) \right)^{-1} \cdot \mathlarger{\int} \delta G \right).
\end{array}
$$

By Proposition~\ref{prop:differentialX}, the map $dX_n$ is equal to $ \tau _n \circ ( X_n , \text{id} )$, where id denotes the identity map on $E_n$. We associate to a locally analytic function $f$ the Legendre function associated to the epigraph of $f$, $\Lambda (f) \, : \, \mathbb{R} \cup \{ \infty \} \longrightarrow \mathbb{R} \cup \{ \infty \}$ $($see \cite[Section~3.2]{carova14} for an explicit definition$)$. Also, we define
$$ \Lambda(f) _{\geq 2} (x) = \inf \limits _{y \geq 0} \left(\Lambda  \! \left(f \right) \! (x+y) {-} 2y \right).$$

\begin{lem}
\label{lem:Lambda}
Let $x \in \mathbb{R}$ such that $x < -2\dfrac{\log p}{p-1}$, then $\Lambda \! \left( X_n \right)_{\geq 2} \! (x)  < x$.
\end{lem}

\begin{proof}
One checks easily that $\Lambda (\text{id})(x) = x$ and $\Lambda (\tau _n )(x) \geq 0$ for all $x \in \mathbb{R}_+^*$. Applying \cite[Proposition~2.5]{carova15}, we get $\Lambda \! \left( X_n \right)_{\geq 2} (x)  \leq  2 \! \left( x + \dfrac{\log p}{p-1}   \right)$ if $x \leq - \dfrac{\log p}{p-1}  $. Therefore, $\Lambda  \! \left( X_n \right)_{\geq 2} \! (x)  < x$ if $ x <  -2\dfrac{\log p}{p-1}  $.
\end{proof}

\begin{prop}
\label{prop:precisionlemma}
Let $B_{E_n} \! \left( \delta \right)$ $($resp. $B_{F_n} \! \left( \delta \right) )$ be the closed ball in $E_n$ $($resp. in $F_n)$ of center $0$ and radius $\delta$. Under the assumption of Lemma~\ref{lem:isometry}, we have for all $ \delta <  p^{\frac{-2}{p-1}}$,
$$ X_n \! \left( G + B_{E_n} \! \left( \delta \right) \right) = X_n \! \left( G \right) + B_{F_n} \! \left( \delta \right).$$
\end{prop}

\begin{proof}
As a direct consequence of \cite[Proposition~3.12]{carova14} and Lemma~\ref{lem:Lambda}, we have the following formula
$$ X_n \! \left( G + B_{E_n} \! \left( \delta \right) \right) = X_n \! \left( G \right) + dX_n \! \left(G\right) \! \left( B_{E_n} \! \left( \delta \right)\right),$$ for all $\delta <  p^{\frac{-2}{p-1}}$.
The result follows from Lemma~\ref{lem:isometry}.
\end{proof}
We end this section by giving a proof of Theorem~\ref{thm:mainthm}.
\begin{proof}[Correctness proof of Theorem~\ref{thm:mainthm}]
Let $G, \mathrm{f}, n$ and $g$ be the output of Algorithm~\ref{algo:DiffSolve}. We first prove by induction on $n \geq 1$ the following equation
$$ H_\mathrm{f} \! \left( X_n \right) \cdot X_n^{'} = G \mod {(t^{n} , p^M)}.$$ Let $m$ be a positive integer and $n = 2m+1$.
Let $e_ m = {G - H_\mathrm{f} \!  \left( X_m \right) \cdot  X'_m}$.
From the relation $$ X_n = X_m + \left( H_\mathrm{f} \! \left( X_m \right) \right)^{-1}  \mathlarger{\int} {e_m\, dt}\; \mod {(t^{n+1} , p^M)}\,,$$
we derive the two formulas
\begin{equation}
\label{eq:proof1}
H_\mathrm{f} \! \left( X_m \right) \cdot X_n =  H_\mathrm{f} \! \left( X_m \right) \cdot X_m +  \mathlarger{\int} {e_m\, dt}\; \mod {(t^{n+1} , p^M)}
\end{equation}
and
\begin{align*}
H_\mathrm{f} \! \left( X_m \right) \cdot X'_n &=  H_\mathrm{f} \! \left( X_m \right) \cdot X'_m +  \left ( H_\mathrm{f} \! \left( X_m \right) \right)' \cdot \left( X_m - X_n\right) + e_m \mod {(t^{n} , p^M)}\,
\\ & = G + \left( H_\mathrm{f} \! \left( X_m \right)\right)' \cdot \left( X_m - X_n\right)\mod {(t^{n} , p^M)}\,
\\ & = G - \left( H_\mathrm{f} \! \left( X_m \right)\right) ' \cdot \left( H_\mathrm{f} \! \left( X_m \right) \right)^{-1}  \mathlarger{\int} {e_m\, dt \mod {(t^{n} , p^M)}}\,.
\end{align*}
Using the fact that the first $m$ coefficients of $e_m$ vanish, we get
\begin{align}
\label{eq:proof2}
H_\mathrm{f} \! \left( X_n \right)  \cdot X'_n  = H_\mathrm{f} \! \left( X_m \right) \cdot X'_n + dH _\mathrm{f} \! \left(X_m \right) \! \left( \left( H_\mathrm{f} \! \left( X_m \right) \right)^{-1}  \mathlarger{\int} {e_m\, dt} \right) \cdot X'_m \mod {(t^{n} , p^M)}\,.
\end{align}
In addition, one can easily verifies
$$
dH _\mathrm{f} \! \left(X_m \right) \! \left( \left( H_\mathrm{f} \! \left( X_m \right) \right)^{-1}  \mathlarger{\int} {e_m\, dt} \right)\cdot X'_m = \left( H_\mathrm{f} \! \left( X_m \right)\right)' \cdot \left( H_\mathrm{f} \! \left( X_m \right) \right)^{-1}  \mathlarger{\int} {e_m\, dt}
$$
Hence, Equation~\eqref{eq:proof2} becomes
$$
H_\mathrm{f} \! \left( X_n \right)  \cdot X'_n = G \mod {(t^{n} , p^M)}.
$$
Now, we define $ G_n = H_\mathrm{f} \! \left( X_n \right)  \cdot X'_n$ so that we have $X_n = X_n \! \left( G_n \right)$ and $\Vert G - G_n \Vert _{F_{n}} \leq p^{-M}$. Therefore, $ \Vert G - G_n \Vert _{E_{n}} \leq p^{-M+\lfloor \log _p(n) \rfloor}$. By Proposition~\ref{prop:precisionlemma}, we have that
$$X_n \! \left( G_n \right) = X_n \! \left( G \right) \mod {(t^{n+1} , p^N)}.$$
Thus $ X_n = X_n \! \left( G \right) \mod {(t^{n+1} , p^N)}$.
\end{proof}

\section{Experiments}
Using an implementation of both Algorithm~\ref{algo:DiffSolve} and the
\textsc{half-gcd} variant given in~\cite{thome03} with the \textsc{magma}
computer algebra system~\cite{magma}, we compute the first $g$ components $\sigma _1 , \ldots \sigma _g$ of the associated rational representation for the multiplication by an integer $\ell$ for Jacobians of genus $2$ and $3$, timings are  detailed in Section~\ref{subsec:timings}. The calculations are done at $p$-adic precision $O(p^M)$ with $M = 1+\lfloor \log _p(2g \ell) \rfloor$.
In addition to our implementation, we make use of Couveignes and Ezome's Algortihm \cite{couezo15} to compute explicit isogenies between Jacobians of genus two curves over a finite extension of $\mathbb{F}_p$ by passing through a finite extension of $\mathbb{Q}_p$. A complete example is given below.

\subsection{An example}
\label{sbsec:example}
We consider the genus two curve given by
\begin{math}
  C_1/\mathbb{F}_{19}:\, y^2 = x^5 + 16\, x^4 + 11 \, x^3 + 3\, x^2 + 5\, x + 17\,.
\end{math}
Let $J(C_1)$ its Jacobian and $\ell$ be a prime number different from $19$. We
look for a maximal isotropic subgroup $V$ of $J(C_1)[\ell]$ which is invariant
by the Frobenius endomorphism. Such a group is found for $\ell = 11$,
therefore an $(11 , 11)$-isogeny over $\mathbb{F}_{19}$ exists. Let us compute its rational representation by applying Algorithm~\ref{algo:DiffSolve} to Equation~\eqref{eq:diffsys}. \\
The $p$-adic precision needed to do the calculations is therefore equal to
$1+\lfloor \log _{19}(110) \rfloor = 2$. We first lift $C_1$ over
$\mathbb{Q}_{19}$ as
\begin{multline*}
  \mathcal{C}_1/\mathbb{Q}_{19} : \, y^2 = x^5 + (16 + O(19^2))\, x^4 + (11+
  O(19^2)) \, x^3 +\\ (3 +
  O(19^2))\, x^2 + (5+ O(19^2))\, x + 17 + O(19^2)\,.
\end{multline*}
We lift the subgroup $V$ as $\mathcal{V}$ in a finite extension of $\mathbb{Q}_{19}$ by lifting its two generators. Let $C_2$ $($resp $\mathcal{C}_2)$ be the curve such that $J(C_2)=J(C_1)/V$ $($resp $J(\mathcal{C}_1)/\mathcal{V})$. Using the main algorithm of \cite{couezo15}, we find an equation of $\mathcal{C}_2$,
\begin{multline*}
  \mathcal{C}_2/\mathbb{Q}_{19} : \, y^2 = (2+O(19^2))\,x^5 - ( 176 + O(19^2))
  \, x^4\\ - (100 + O(19^2))\, x^3 + (2546 + O(19^2))\, x^2 - (68 +
  O(19^3))\,x\,,
\end{multline*}
and the normalization matrix being equal to
$$  \begin{pmatrix}
95+ O(19^2) & 233+ O(19^2)\\
155+ O(19^2) &  228+O(19^2)
\end{pmatrix} .$$
The computation of the normalization matrix is done by sending the formal point $$P_1(t)= \left ( t+  O(19^2) , 146 -21 \, t + 179\,t^2 +  O\big(19^2,t^3) \right)  \in \mathcal{C}_1 \left( \mathbb{Q}_{19}\llbracket t \rrbracket\right)$$ to
\begin{multline*}
\Big \{ R_1 = \left(-36  + 353 \,t  + O\big(19^2,t^2), -13  + 326 \, t + O\big(19^2,t^2) \right) , \\ R_2= \left(-129  + 102\, t+ O\big(19^2,t^2) , -47  + 2 \, t + O\big(19^2,t^2) \right)  \Big \}
\end{multline*}
in $\mathcal{C}_2\left( \mathbb{Q}_{19}\llbracket t \rrbracket\right) ^{(2)}$.
We can therefore choose $X_0 = \left(  O(19^2) , 146 + O(19^2)\right)$ as an initial condition for the differential equation, then send it to the point $\left(O(19^2) , O(19^2)\right)$ by making the change of variables $X(t) \leftarrow X(t)-X_0$ . Using the equation of the curve $\mathcal{C}_1$, we compute the $y$-coordinate of $P_1(t)$ modulo $(19^2, t^{111})$, then we compute $G \mod (19^2,t^{111})$. \\
A call from Algorithm~\ref{algo:DiffSolve},  gives the series $x_1(t), x_2(t), y_1(t)$ and $y_2(t)$ modulo $(19^2,t^{111})$. For instance, the first $21$ terms of $x_1(t)$ and $x_2(t)$ are given by
\begin{multline*}
x_1(t)= -36 - 8t - 58t^2 - 90t^3 - 90t^4 - 145t^5 - 124t^6 - 107t^7 - 13t^8 - 114t^9 + 154t^{10} + 129t^{11}  + 88t^{12} \\ + 103t^{13} - 22t^{14} - 147t^{15} -
    178t^{16} + 168t^{17} + 144t^{18} - 166t^{19} - 77t^{20} + O(19^2,t^{21})
\end{multline*}
and
\begin{multline*}
x_2(t) = -129 + 102t + 100t^2 + 94t^3 + 45t^4 + 91t^5 + 29t^6 + 137t^7 - 132t^8 - 52t^9 + 51t^{10} + 150t^{11} + 80t^{12} \\ + 90t^{13} - 124t^{14} - 163t^{15} +
    90t^{16} + 102t^{17} + 55t^{18} + 44t^{19} + 23t^{20} + O(19^2,t^{21}).
\end{multline*}
Applying the \textsc{half-gcd} algorithm to the series $x_1(t) + x_2(t), x_1(t) \cdot x_2(t), (y_2(t) - y_1(t))/(x_2(t) - x_1(t))$ and $ (y_1(t) \cdot x_2(t) - y_2(t) \cdot x_1(t)) / (x_2(t) - x_1(t))$ modulo $19$, we recover the rational functions $\sigma_1, \sigma_2, \alpha_1$ and $\alpha_2$. For instance, the numerator $N$ of $-\sigma_1$ is given by
\begin{multline*}
N= x^{20} + 8x^{19} + 12x^{18} + 4x^{17} + 16x^{16} + 2x^{15} + 18x^{14} +
    2x^{13} + 18x^{12} + 16x^{11} + 13x^{10} \\ + 6x^9 + 5x^8 + 10x^7
    + 5x^6 + 10x^5 + 9x^4 + 17x^3 + 18x^2 + 1
\end{multline*}
and its denominator $D$ is equal to
\begin{multline*}
D= 12x^{21} + 11x^{20} + 18x^{19} + 14x^{18} + 13x^{16} + 18x^{15} + 8x^{14}
    + 5x^{13} + 13x^{12} + 16x^{11} + 2x^{10} \\ + 5x^9 + 3x^8 + 4x^7
    + 6x^6 + 5x^5 + 18x^4 + 11x^3 + 16x^2 + 9x + 16.
\end{multline*}

\subsection{Timings}
\label{subsec:timings}
We use an implementation in \textsc{magma} of Algortihm~\ref{algo:DiffSolve}
to compute the components $\sigma _1, \ldots , \sigma _g$ of the rational
representation of the multiplication by $\ell$ map in $\mathbb{F}_7$ for Jacobians of hyperelliptic curves of genus $2$ and $3$ for some $\ell \in \{ 0, \ldots,
461\}$. Results are detailed on Figure~\ref{fig:timingsgf2}. The base ring of all our computations does not change, it is always $ \mathbb{Z}/7^\lambda \mathbb{Z}$ for $\lambda = 1 + \lfloor \log _7( 2g \ell^2) \rfloor$, so the
timings for $g=3$ are significantly larger than those of $g=2$ by a small constant factor.

\begin{figure}[htbp]
  \centering
  \includegraphics[width=0.55\textwidth, angle=-90]{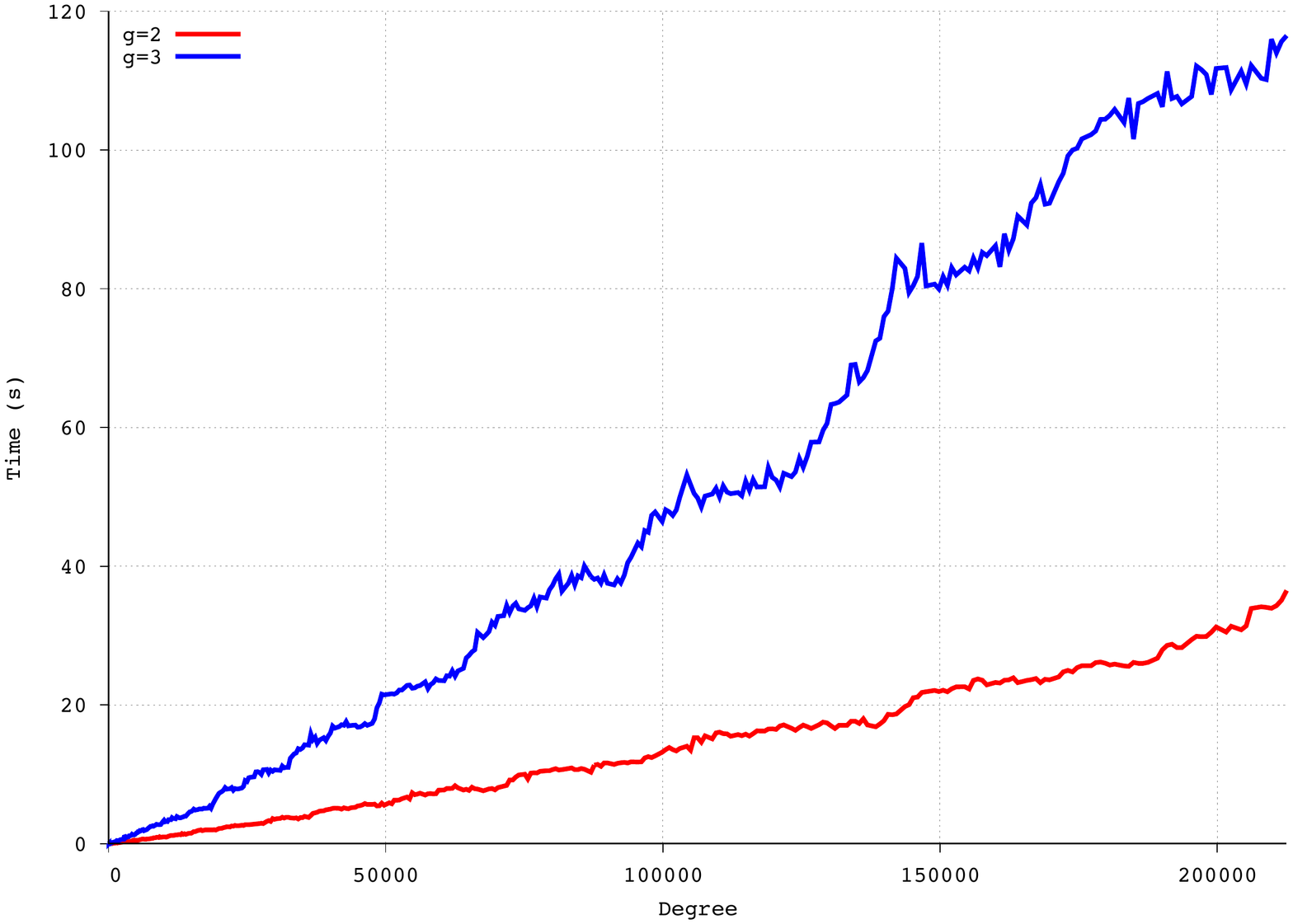}
  \caption{Isogeny computations for Jacobians of genus $2$ and $3$ curves in $\mathbb{F}_7$.}
  \label{fig:timingsgf2}
\end{figure}

\bibliographystyle{alphaabbr}
\bibliography{synthbib}

\end{document}